\documentclass[reqno, 12pt]{amsart}
\usepackage{amssymb}
\usepackage{amsfonts}
\usepackage{enumerate}
\usepackage{cite}
\usepackage{amsmath}
\usepackage{amsfonts,amssymb,latexsym}
\usepackage[OT4]{fontenc}
\usepackage[cp1250]{inputenc}
\usepackage{mathtools}
\usepackage{verbatim}
\usepackage{wasysym}

\newtheorem{theorem}{Theorem}[section]
\newtheorem{lemma}[theorem]{Lemma}
\newtheorem{proposition}[theorem]{Proposition}
\newtheorem{corollary}[theorem]{Corollary}
\theoremstyle{definition}
\newtheorem{definition}[theorem]{Definition}

\theoremstyle{remark}
\newtheorem{remark}[theorem]{Remark}

\DeclareMathOperator{\diam}{diam}

\DeclareMathOperator{\Fix}{Fix}

\DeclarePairedDelimiter\floor{\lfloor}{\rfloor}

\begin{document}
\title[Uniform asymptotic regularity]{Applications of uniform asymptotic
regularity to fixed point theorems}
\author[S. Borzdy\'{n}ski]{S\l awomir Borzdy\'{n}ski}
\author[A. Wi\'{s}nicki]{Andrzej Wi\'{s}nicki}

\begin{abstract}
We show that there are no nontrivial surjective uniformly asymptotically
regular mappings acting on a metric space and derive some consequences of
this fact. In particular, we prove that a jointly continuous left amenable
or left reversible semigroup generated by firmly nonexpansive mappings on a
bounded $\tau $-compact subset of a Banach space has a common fixed point,
and give a qualitative complement to the Markov-Kakutani theorem.
\end{abstract}

\address{S\l awomir Borzdy\'{n}ski, Institute of Mathematics, Maria Curie-Sk%
\l odowska University, 20-031 Lublin, Poland}
\email{slawomir.borzdynski@gmail.com}
\address{Andrzej Wi\'{s}nicki, Department of Mathematics, Rzesz\'{o}w
University of Technology, Al. Powsta\'{n}c\'{o}w Warszawy 12, 35-959 Rzesz%
\'{o}w, Poland}
\email{awisnicki@prz.edu.pl}
\maketitle

\section{Introduction}

The notion of asymptotic regularity was introduced by Browder and Petryshyn
in \cite{BrPe} in connection with the study of fixed points of nonexpansive
mappings. Recall that a mapping $T:M\rightarrow M$ acting on a metric space $%
(M,d)$ is asymptotically regular if
\begin{equation*}
\lim_{n\rightarrow \infty }d(T^{n}x,T^{n+1}x)=0
\end{equation*}%
for all $x\in M.$ Krasnoselskii \cite{Kra} proved that if $K$ is a compact
convex subset of a uniformly convex Banach space and if $T:K\rightarrow K$
is nonexpansive (i.e., $\Vert Tx-Ty\Vert \leq \Vert x-y\Vert $ for $x,y\in K$%
), then for any $x\in K$ the sequence $\{(\frac{1}{2}I+\frac{1}{2}T)^{n}x\}$
converges to a fixed point of $T.$ He used the fact that the averaged\
mapping $\frac{1}{2}(I+T)$ is asymptotically regular. Subsequently,
Ishikawa\ \cite{Is} proved that if $C$ is a bounded closed convex subset of
a Banach space $E$ and $T:C\rightarrow C$ is nonexpansive, then the mapping $%
T_{\alpha }=(1-\alpha )I+\alpha T$ is asymptotically regular for each $%
\alpha \in (0,1).$ Independently, Edelstein and O'Brien \cite{EdOb} showed
that $T_{\alpha }$ is uniformly asymptotically regular over $x\in C$, and
Goebel and Kirk \cite{GoKi3} proved that the convergence is uniform with
respect to all nonexpansive mappings from $C$ into $C$. Other examples of
asymptotically regular mappings are given by the result of Anzai and
Ishikawa \cite{AnIs} -- if $T$ is an affine mapping acting on a bounded
closed convex subset of a locally convex space $X$, then $T_{\alpha
}=(1-\alpha )I+\alpha T$ is uniformly asymptotically regular.

Another important class of nonlinear operators are firmly nonexpansive
mappings, i.e. mappings with the property
\begin{equation*}
\left\Vert {Tx-Ty}\right\Vert \leq \left\Vert {\alpha (x-y)+(1-\alpha
)(Tx-Ty)}\right\Vert
\end{equation*}%
for all $x,y\in C$ and $\alpha \in (0,1)$. Their origin is associated with
the study of maximal monotone operators in Hilbert spaces. Reich and Shafrir
\cite{ReSh} proved that every firmly nonexpansive mapping $T:C\rightarrow C$
is asymptotically regular provided $C$ is bounded. For a thorough treatment
of firmly nonexpansive mappings we refer the reader to \cite{ALM}.

Quite recently, Bader, Gelander and Monod \cite{BGM} proved a remarkable
fixed point theorem for affine isometries in $L$-embedded Banach spaces
preserving a bounded subset $A$ and showed its several applications. In
particular, they simplified the Losert's proof and simultaneously obtained
the optimal solution to the \textquotedblleft derivation
problem\textquotedblright\ studied since the 1960s. The present paper is
partly motivated by the Bader-Gelander-Monod theorem. The paper \cite{BrRe}
is also relevant at this point. Our basic observation is Proposition \ref%
{surjectiveImplId} -- there are no nontrivial surjective uniformly
asymptotically regular mappings. It appears to be quite useful in fixed
point theory of nonexpansive and affine mappings.

The well-known Day generalization of the Markov-Kakutani fixed point theorem
asserts that a semigroup $S$ is left amenable if and only if, whenever $S$
acts as affine continuous mappings on a nonempty compact convex subset $C$
of a Hausdorff locally convex space, $S$ has a common fixed point in $C$. In
1976, T.-M. Lau \cite{Lau1} (see also \cite[Question 1]{LaZh}) posed the
problem of characterizing left amenability of a semigroup $S$ in terms of
the fixed point property for nonexpansive mappings. Consider the following
fixed point property:

\begin{enumerate}
\item[$(F_{\ast })$:] \textit{Whenever $\mathcal{S}=\{T_{s}:s\in S\}$ is a
representation of $S$ as norm-nonexpansive mappings on a non-empty weak$%
^{\ast }$ compact convex set $C$ of a dual Banach space $E$ and the mapping $%
(s,x)\rightarrow T_{s}(x)$ from $S\times C$ to $C$ is jointly continuous,
where $C$ is equipped with the weak$^{\ast }$ topology of $E$, then there is
a common fixed point for }$\mathcal{S}$\textit{\ in $C$.}
\end{enumerate}

It is not difficult to show (see, e.g., \cite[p. 528]{LaTa2}) that property $%
(F_{\ast })$ implies that $S$ is left amenable. Whether the converse is true
is still an open problem though a few partial results are known. Lau and
Takahashi \cite[Theorem 5.3]{LaTa} proved that the answer is affirmative if $%
C$ is separable. Recently, the authors of the present paper were able to
prove that commutative semigroups have $(F_{\ast })$ property (see \cite[%
Theorem 3.6]{BoWi}). We show in Section 3 that a jointly continuous left
amenable or left reversible semigroup generated by firmly nonexpansive
mappings acting on a bounded $\tau $-compact subset of a Banach space has a
common fixed point. This is a partial result but the techniques developed in
this paper may lead to a solution of the original problem too.

In Section 4 we apply our techniques to the case of commutative semigroups
of affine mappings and give the qualitative complement to the
Markov-Kakutani theorem. In spite of its simplicity, the result seems to be
generally new even for the weak topology, see Remark \ref{weak}.

\section{Preliminaries}

Let $(M,d)$ be a metric space. In this paper we focus on uniformly
asymptotically regular self-mappings on $M$ (abbr. UAR), i.e., mappings $%
T:M\rightarrow M$ which satisfy the condition

\begin{equation*}
\lim_{n\rightarrow \infty }\sup_{x\in M}\ d({T^{n+1}x,T^{n}x)}=0.
\end{equation*}

The following observation is our basis for using uniform asymptotic
regularity to generate fixed points.

\begin{lemma}
\label{lemma1}A surjective mapping $T:M\rightarrow M$ with the property
\begin{equation*}
\inf_{n}\sup_{x\in M}\ d({T^{n+1}x,T^{n}x)}=0,
\end{equation*}%
is the identity mapping.
\end{lemma}

\begin{proof}
Take $\varepsilon >0$. The assured property gives us a positive integer $n$
such that%
\begin{equation*}
\forall _{x\in M}\ d({T^{n+1}x,T^{n}x)}<\varepsilon .
\end{equation*}%
Take arbitrary $y\in M$. From the surjectivity of $T,$ there exists $x\in M$
such that $T^{n}x=y$. Then the above inequality turns out to be $d({y,Ty)}%
<\varepsilon $. Since $\varepsilon >0$ is arbitrary, we have $Ty=y$ and
consequently, $T=Id$.
\end{proof}

\begin{proposition}
\label{surjectiveImplId}There are no nontrivial surjective UAR mappings.
\end{proposition}

Notice that uniformity assumption is necessary as $Tx=x^{2}$ defined on $%
[0,1]$ meets the demands of the above corollary, except it is merely
asymptotically regular and obviously not the identity.

Let $C$ be a convex subset of a Banach space $X$ and $\alpha \in (0,1).$ For
brevity and symmetry, we denote $1-\alpha $ by $\bar{\alpha}$. A mapping of
the form%
\begin{equation}
T=\alpha I+\bar{\alpha}S
\end{equation}%
is called ($\alpha $-)averaged nonexpansive (resp. affine) provided that $%
S:C\rightarrow C$ is nonexpansive (resp. affine). The term \textquotedblleft
averaged mapping\textquotedblright\ was coined in \cite{BBR}. Whenever a $%
\tau $-topology on $C$ is mentioned it is assumed to be Hausdorff. A set $C$
is called $T$-invariant if $T(C)\subset C$.

\section{Applications to nonexpansive mappings}

We start by giving a few statements which result in the uniform asymptotic
regularity. The first one was proved by Edelstein and O'Brien \cite[Theorem 1%
]{EdOb} (see also \cite[Theorem 2]{GoKi3}).

\begin{lemma}
\label{UAR}An averaged nonexpansive mapping $T:C\rightarrow C$ defined on a
convex and bounded subset $C$ of a Banach space is UAR.
\end{lemma}

Recall that in a Hilbert space a mapping $T:C\rightarrow C$ is firmly
nonexpansive iff $T=\frac{1}{2}(I+S)$ for a nonexpansive mapping $S.$ In
general, there is no such relation. However, a counterpart of Lemma \ref{UAR}
can be drawn for firmly nonexpansive mappings by combining a method of \cite[%
Theorem 9.4]{GoKi} with \cite[Theorem 1]{ReSh}.

\begin{lemma}
\label{firmlyImpliesUniformly}A firmly nonexpansive mapping $T:C\rightarrow
C $ defined on a bounded subset of a Banach space $E$ is UAR.
\end{lemma}

\begin{proof}
Let $C$ be a bounded subset of $E$ and put
\begin{equation*}
E^{\star }=\{S\in \,E^{C}:\left\Vert {S}\right\Vert <\infty \},
\end{equation*}%
where
\begin{equation*}
\left\Vert {S}\right\Vert =\sup \{\left\Vert Sx\right\Vert :x\in C\}
\end{equation*}%
is the uniform norm. Let $C^{\star }=C^{C}.$ It is known, that $E^{\star }$
is a Banach space. Also, $C^{\star }$ inherits boundedness from $C$:%
\begin{equation*}
\forall _{T\in C^{\star }}\sup_{x\in C}\left\Vert {Tx}\right\Vert \leq
M\implies \sup_{T\in C^{\star }}\left\Vert {T}\right\Vert \leq M.
\end{equation*}%
For $T\in C^{\star }$ define $\Psi _{T}:C^{\star }\rightarrow C^{\star }$ by
$\Psi _{T}\,S=T\circ S$. If $T$ is firmly nonexpansive then%
\begin{equation*}
\left\Vert {(\Psi _{T}S-\Psi _{T}N)x}\right\Vert =\left\Vert {T(Sx)-T(Nx)}%
\right\Vert \leq \left\Vert {(\alpha (S-N)+\bar{\alpha}(\Psi _{T}S-\Psi
_{T}N))x}\right\Vert
\end{equation*}%
and, by taking suprema, we conclude that $\Psi _{T}$ is also firmly
nonexpansive.

Now, from \cite[Theorem 1]{ReSh} we have%
\begin{equation*}
\lim_{n}\left\Vert {\Psi _{T}^{n+1}S-\Psi _{T}^{n}S}\right\Vert
=\lim_{n}\left\Vert {\frac{\Psi _{T}^{n}S}{n}}\right\Vert =0,
\end{equation*}%
where the last equality follows from boundedness of $C^{\star }$. Taking $S$
as the identity on $C$, and unwiding definitions of $\left\Vert {\cdot }%
\right\Vert $ and $\Psi _{T}$, we get%
\begin{equation*}
\lim_{n}\sup_{x\in C}\left\Vert {T^{n+1}x-T^{n}x}\right\Vert =0
\end{equation*}%
which is the desired conclusion.
\end{proof}

Now from the ingredients above we have instantly the following result.

\begin{theorem}
\label{fixNonEmpty} If for a firmly nonexpansive (resp. averaged
nonexpansive) mapping $T:C\rightarrow C$ defined on a subset (resp. convex
subset) of a Banach space there exists a nonempty bounded set $D\subset C$
such $T(D)=D$, then $D\subset Fix\,T$. In particular, $\Fix T$ is nonempty.
\end{theorem}

\begin{corollary}
Suppose $C$ is a closed convex subset of a Banach space and $T:C\rightarrow C
$ is nonexpansive. If for some $\alpha \in (0,1)$, the mapping $F=(I-\alpha
\,T)^{-1}\bar{\alpha}I$ is surjective on a (nonempty) bounded set $D\subset C$, then $T$
has a fixed point.
\end{corollary}

\begin{proof}
It follows from the fact that $F$ is firmly nonexpansive (see \cite[p. 122]%
{GoKi}) and $\Fix F=\Fix T.$
\end{proof}

Below we use our technique to obtain fixed point theorems for semigroups of
uniformly asymptotically regular mappings. But first let us introduce some
notions.

\begin{definition}
\label{representation} We say that $\mathcal{S}=\left\{ T_{s}:s\in S\right\}
$ is a representation of $S$ on a topological space $(C,\tau )$ if for each $%
s\in S,$ $T_{s}$ is a mapping from $C$ into $C$ and $T_{st}x=T_{s}(T_{t}x)$
for every $s,t\in S$ and $x\in C$.
\end{definition}

Since, due to Lemma \ref{lemma1}, surjectivity is intertwined into all
results of this paper, the following definition seems natural:

\begin{definition}
We say that the representation $\mathcal{S}$ on the space $C$ is
subsurjective if there is a (nonempty) set $D\subset C$ such that every $%
T\in \mathcal{S}$ is surjective on $D$.
\end{definition}

There is a fundamental example of such representations.

\begin{lemma}
\label{commutingSubsu}Every commuting semigroup of continuous mappings
represented on a compact space is subsurjective on it.
\end{lemma}

\begin{proof}
Using the Kuratowski-Zorn lemma, we get a minimal compact $\mathcal{S}$%
-invariant set $D$. Now, picking any $T\in \mathcal{S}$, we notice that for
every $P\in \mathcal{S}$,
\begin{equation*}
P(D)\subset D\implies TP(D)\subset T(D)\implies PT(D)\subset T(D),
\end{equation*}%
meaning that $T(D)$ is $\mathcal{S}$-invariant. It is also compact so, by
the minimality of $D$, $T(D)=D$.
\end{proof}

Let $S$ be a semitopological semigroup, i.e., a semigroup with a Hausdorff
topology such that the mappings $S\ni s\rightarrow ts$ and $S\ni
s\rightarrow st$ are continuous for each $t\in S.$ Notice that every
semigroup can be equipped with the discrete topology and then it is called a
discrete semigroup. The semigroup $S$ is said to be left reversible if any
two closed right ideals of $S$ have a non-void intersection. In this case $%
(S,\leq )$ is a directed set with the relation $a\leq b$ iff $\left\{
a\right\} \cup \overline{aS}\supset \left\{ b\right\} \cup \overline{bS}.$

Let $\ell ^{\infty }(S)$ be the Banach space of bounded real-valued
functions on $S$ with the supremum norm. For $s\in S$ and $f\in \ell
^{\infty }(S),$ we define the element $l_{s}f$ in $\ell ^{\infty }(S)$ by%
\begin{equation*}
l_{s}f(t)=f(st)
\end{equation*}%
for every $t\in S.$ Denote by $C_{b}(S)$ the closed subalgebra of $\ell
^{\infty }(S)$ consisting of continuous functions and let $LUC(S)$ be the
space of bounded left uniformly continuous functions on $S,$ i.e., all $f\in
C_{b}(S)$ such that the mapping $S\ni s\rightarrow l_{s}f$ from $S$ to $%
C_{b}(S)$ is continuous when $C_{b}(S)$ has the sup norm topology. Note that
if $S$ is a topological group, then $LUC(S)$ is equivalently the space of
bounded right uniformly continuous functions on $S$ (see \cite{HeRo}). A
semigroup $S$ is called left amenable if there exists a left invariant mean
on $LUC(S)$, i.e., a functional $\mu \in LUC(S)^{\ast }$ such that $%
\left\Vert \mu \right\Vert =\mu (I_{S})=1$ and%
\begin{equation*}
\mu (l_{s}f)=\mu (f)
\end{equation*}%
for each $s\in S$ and $f\in LUC(S).$ If $S$ is discrete and left amenable
then $S$ is left reversible. In general there is no such relation (see \cite[%
p. 335]{HoLa}).

Let us introduce one more temporary definition.

\begin{definition}
If in the markings as stated in the Definition \ref{representation}, $%
\mathcal{S}$ is a left reversible or left amenable semitopological semigroup
and the mapping
\begin{equation*}
S\times C\ni (s,\,x)\rightarrow T_{s}x
\end{equation*}%
is jointly continuous, then we say that the semigroup is ($\tau $-)properly
represented.
\end{definition}

Since it is known that, notably, commuting semigroups are left amenable we
can now conveniently state a generalization of Lemma \ref{commutingSubsu}.

\begin{lemma}
\label{propRep} Semigroups of mappings, properly represented on a compact
space are subsurjective.
\end{lemma}

\begin{proof}
Just use \cite[Lemma 3.4]{LaZh} in the case of left reversible semigroups
and \cite[Lemma 5.1]{LaTa} for left amenable ones. Note also that in both
cases the resulting set is compact.
\end{proof}

Now, after we presented some decent examples of subsurjectivity, we
formulate the main result of this section.

\begin{theorem}
\label{uar_fixed} Suppose $(M,d)$ is a metric space. Then a subsurjective
semigroup $\mathcal{S}$ generated by UAR mappings on $M$ has a common fixed
point.
\end{theorem}

\begin{proof}
It follows from subsurjectivity that there exists $D\subset M$ such that $%
T_{s}(D)=D$ for $s\in S$. Now, Proposition \ref{surjectiveImplId} yields
that every generator of $\mathcal{S}$ is the identity on $D$. It follows
that $T_{s}x=x$ for every $x\in D$ and $s\in S$.
\end{proof}

Combining the above theorem with Lemma \ref{firmlyImpliesUniformly} we
obtain the following corollary.

\begin{corollary}
\label{propFirmlyImplFP} Let $\mathcal{S}$ be a subsurjective semigroup on a
bounded subset $C$ of a Banach space whose generators consist of firmly
nonexpansive mappings. Then $\mathcal{S}$ has a common fixed point.
\end{corollary}

Note that a special case of the above corollary gives a partial answer to
the problem of A. T.-M. Lau \cite[Question 1]{LaZh} (as described in the
Introduction) for semigroups generated by firmly nonexpansive mappings.

\begin{corollary}
Let $\mathcal{S}=\{T_{s}:s\in S\}$ be a representation of a left amenable
semigroup on $w^{\ast }$-compact subset $C$ of a Banach space generated by
firmly nonexpansive mappings such the mapping $S\times C\ni
(s,\,x)\rightarrow T_{s}x$ is jointly continuous. Then $\mathcal{S}$ has a
common fixed point.
\end{corollary}

Assuming proper representability instead of just subsurjectivity, we get the
result stronger than Theorem \ref{uar_fixed}.

\begin{lemma}
\label{CFP}Let $\mathcal{S}$ be a properly represented semigroup generated
by UAR mappings defined on a metric space $M$. Then every compact $\mathcal{S%
}$-invariant subset of $M$ contains a fixed point of $\mathcal{S}$.
\end{lemma}

\section{Applications to affine mappings}

Recall that a mapping $T:C\rightarrow C$ is said to be affine if
\begin{equation*}
T(\alpha x+\beta y)=\alpha Tx+\beta Ty
\end{equation*}
for every $x,y\in C$ and $\alpha ,\beta \geq 0,\alpha +\beta =1$.

The famous Markov-Kakutani theorem states that if $C$ is a nonempty compact
and convex subset of a (Hausdorff) linear topological space $X$ and $\left\{
T_{i}\right\} _{i\in I}$ a commutative family of continuous affine mappings
of $C$ into $C$, then there exists a common fixed point $x\in C$: $T_{i}x=x$
for every $i\in I$. Furthermore, $\bigcap_{i\in I}\Fix T_{i}$ is compact and
convex, and thus is an absolute retract provided $C$ is a metrizable subset
of a locally convex space (see, e.g., \cite[Theorem 7.7.5]{GrDu}). In
particular, then, there exists a (continuous) retraction $r:C\rightarrow
\bigcap_{i\in I}\Fix T_{i}.$ A natural question arises whether we can also
obtain an affine retraction onto $\bigcap_{i\in I}\Fix T_{i}$. To the best
of our knowledge only very partial results regarding almost periodic actions
are in the literature (see \cite{Ph}). In this section we prove a general
result of this kind. To this end we need the following lemma, see \cite[%
Lemma 1]{AnIs}, \cite[Theorem 3]{HoKu}.

\begin{lemma}
\label{avg_lin_asym}An averaged affine self-map defined on a convex and
bounded subset of a Banach space is UAR.
\end{lemma}

\begin{proof}
Let $T:C\rightarrow C$ be an $\alpha $-averaged of an affine mapping $S$.
Notice that from the affinity we get a quite compact binominal expansion:%
\begin{equation*}
T^{n}=(\alpha I+\bar{\alpha}S)^{n}=\sum_{k=0}^{n}{\binom{n}{k}}\alpha
^{(n-k)}\bar{\alpha}^{k}S^{k}.
\end{equation*}%
With the notation ${\binom{n}{n+1}}={\binom{n}{-1}}=0,$ we have%
\begin{equation*}
T(T^{n})=\alpha T^{n}+\sum_{k=0}^{n}{\binom{n}{k}}\alpha ^{(n-k)}\bar{\alpha}%
^{k+1}S^{k+1}=\alpha T^{n}+\sum_{k=0}^{n+1}{\binom{n}{k-1}}\alpha ^{(n-k+1)}%
\bar{\alpha}^{k}S^{k}.
\end{equation*}%
Then

\begin{equation*}
T^{n}-T^{n+1}=\bar{\alpha}T^{n}-\sum_{k=0}^{n+1}{\binom{n}{k-1}}\alpha
^{(n-k+1)}\bar{\alpha}^{k}S^{k}=\sum_{k=0}^{n+1}{n\brack k}S^{k},
\end{equation*}%
where
\begin{equation*}
{n\brack k}={\binom{n}{k}}\alpha ^{(n-k)}\bar{\alpha}^{k+1}-{\binom{n}{k-1}}%
\alpha ^{(n-k+1)}\bar{\alpha}^{k}.
\end{equation*}%
We get that ${n\brack k}$ is nonnegative up to $m=\floor{\bar{\alpha}(n+1)}$%
. Using the identity%
\begin{equation*}
{n\brack k}+{n\brack k+1}=-{\binom{n}{k-1}}\alpha ^{(n-k+1)}\bar{\alpha}^{k}+%
{\binom{n}{k+1}}\alpha ^{(n-k+1)}\bar{\alpha}^{k+2}
\end{equation*}%
and having in mind boundary cases, we get by telescoping,%
\begin{equation*}
\sum_{k=0}^{m}{n\brack k}={\binom{n}{m}}\alpha ^{(n-m)}\bar{\alpha}%
^{m+1}=a_{n}\geq 0.
\end{equation*}%
Also, since our coefficents are negative above $m$, we get similarly
\begin{equation*}
\sum_{k=m+1}^{n+1}\left\vert {n\brack k}\right\vert =\left\vert {%
\sum_{k=m+1}^{n+1}{n\brack k}}\right\vert =\left\vert {-{\binom{n}{m}}\alpha
^{(n-m)}\bar{\alpha}^{m+1}}\right\vert =a_{n}.
\end{equation*}%
Since $C$ is convex, it follows that for each $x\in C,$

\begin{equation*}
\sum_{k=0}^{m}\frac{{n\brack k}}{a_{n}}S^{k}x=x_{1}\in C,
\end{equation*}%
\begin{equation*}
\sum_{k=m+1}^{n+1}\frac{\left\vert {n\brack k}\right\vert }{a_{n}}%
S^{k}x=x_{2}\in C.
\end{equation*}%
Notice that
\begin{equation*}
\left\Vert {T^{n}x-T^{n+1}x}\right\Vert =a_{n}\left\Vert x{_{1}-x_{2}}%
\right\Vert \leq a_{n}\diam C.
\end{equation*}%
If $\alpha =\frac{1}{2}$ and $n$ is even, we have from the Stirling's
approximation of the central binomial coefficent
\begin{equation*}
a_{n}=\frac{1}{2^{n+1}}\binom{n}{\frac{1}{2}n}\approx \frac{1}{2^{n+1}}\frac{%
2^{n}}{\sqrt{\pi \frac{n}{2}}}\rightarrow 0.
\end{equation*}

The odd case, $n=2k-1$ gives
\begin{equation*}
a_{n}=\frac{1}{2^{n}}\frac{(2k-1)!}{k!(k-1)!}=\frac{1}{2^{n}}\frac{2k}{2k}%
\frac{(2k-1)!}{k!(k-1)!}=\frac{1}{2^{n+1}}\frac{(2k)!}{(k!)^{2}}\approx
\frac{1}{2^{2k}}\frac{2^{2k}}{\sqrt{\pi k}}\rightarrow 0.
\end{equation*}

This proves that $a_{n}$ vanishes in infinity, and since $C$ is bounded and $%
a_{n}$ does not depend on $x$, we conclude that $T$ is UAR for $\alpha =%
\frac{1}{2}$ (which turns out to be sufficient for our applications). The
case the $\alpha >\frac{1}{2}$ can also be easily proved: define $%
U:C\rightarrow C$ as $\lambda $-averaged of $S$ with $\lambda =2\alpha $.
Notice that $U$ is affine mapping and $T$ is its $1/2$-averaged, so once
again $T$ is UAR. As for the case $\alpha >\frac{1}{2}$ we recommend the
reader to compare \cite[Lemma 1]{AnIs} or \cite[Theorem 3]{HoKu}.
\end{proof}

Notice that in Lemma \ref{CFP} $D$ need not be a convex set. This helps us
to prove our final theorem. But first recall the following variant of
Bruck's theorem (see \cite[Theorem 3]{Br2}).

\begin{theorem}
\label{Bruck}Let $(C,\tau )$ be a topological space and $\mathcal{S}$ a
semigroup of mappings on $C$. Suppose that $\mathcal{S}$ is compact in the
topology of pointwise convergence and each nonempty closed $\mathcal{S}$%
-invariant subset of $C$ contains a fixed point of $\mathcal{S}$. Then there
exists in $\mathcal{S}$ a surjective mapping $r:C\rightarrow \Fix\mathcal{S}$
which is an idempotent: $r\circ r=r.$
\end{theorem}

The following corollary will be sufficient for us.

\begin{corollary}
\label{observation}Let $(C,\tau )$ be a compact topological space and $%
\mathcal{S}$ a semigroup of $\tau $-continuous affine mappings on $C$.
Suppose that each nonempty closed $\mathcal{S}$-invariant subset of $C$
contains a fixed point of $\mathcal{S}$. Then there exists an affine and
surjective idempotent mapping $r:C\rightarrow \Fix\mathcal{S}$.
\end{corollary}

\begin{proof}
Let $\mathcal{\bar{S}}$ denote the closure of $\mathcal{S}$ in $C^{C}$ in
the $\tau $-product topology (i.e., topology of pointwise convergence with
respect to the $\tau $-topology). Notice that if $S,T\in \mathcal{\bar{S}}$,
then there exists nets $\{S_{\alpha }\}_{\alpha \in A},\{T_{\beta }\}_{\beta
\in B}$ of mappings in $\mathcal{S}$ such that $Sx=\tau $-$\lim_{\alpha
}S_{\alpha }x$ and $Tx=\tau $-$\lim_{\beta }T_{\beta }x$ for $x\in C$.
Notice that for every $\alpha $, $\{S_{\alpha }T_{\beta }\}_{\beta }\subset
\mathcal{S}$ and from $\tau $-continuity of $S_{\alpha }$,%
\begin{equation*}
\mathcal{\bar{S}}\ni \tau \text{-}\lim_{\beta }(S_{\alpha }T_{\beta
})=S_{\alpha }(\tau \text{-}\lim_{\beta }T_{\beta })=S_{\alpha }T,
\end{equation*}%
where the above limits are to be understood in the topology of $\tau $%
-pointwise convergence. Hence $\{S_{\alpha }T\}_{\alpha }\subset \mathcal{%
\bar{S}}$ and again,

\begin{equation*}
\mathcal{\bar{S}}=\mathcal{\bar{\bar{S}}}\ni \tau \text{-}\lim_{\alpha
}S_{\alpha }T=S\circ T
\end{equation*}%
which follows from the pointwise definition of the above limit. Thus $%
\mathcal{\bar{S}}$ is a semigroup. Since it is also compact as the closed
subset of the compact set $C^{C}$ (in the $\tau $-product topology) and
consists of affine mappings, we can apply Bruck's theorem to get in $%
\mathcal{\bar{S}}$ an affine idempotent $r$ from $C$ onto $\Fix\mathcal{\bar{%
S}}$ which is our desired mapping since $\Fix\mathcal{\bar{S}=}\Fix\mathcal{S%
}$.
\end{proof}

\begin{theorem}
\label{Gen}Suppose $\mathcal{S}$ is a $\tau $-properly represented semigroup
generated by averaged affine mappings on $C$, where $C$ is a convex and $%
\tau $-compact bounded subset of a Banach space. Then there exists an
affine, surjective idempotent $r:C\rightarrow \Fix\mathcal{S}.$ If,
moreover, elements of $\mathcal{S}$ are locally $\tau $-equicontinuous, then
we can pick $r$ to be also $\tau $-continuous.
\end{theorem}

\begin{proof}
Since from Lemma \ref{avg_lin_asym} generators of $\mathcal{S}$ are UAR, the
existence of an affine, surjective idempotent $r$ follows instantly from
Lemma \ref{CFP} and Corollary \ref{observation}.

Assume now that mappings from $\mathcal{S}$ are locally $\tau $%
-equicontinuous. Pick $x_{0}\in C$, a net $\{T_{\alpha }\}_{\alpha \in
A}\subset \mathcal{S}$ and let $T$ be the limit of this net in the $\tau $%
-product topology. As earlier, let $\bar{A}$ denote the closure of $A$ in
this topology. Pick any $B,\,U\in \tau _{0}$ such that $\bar{B}\subset U$,
where $\tau _{0}$ denotes $\tau $-neighbourhoods of the origin. Then we
conclude from local $\tau $-equicontinuity that
\begin{equation*}
\exists _{V\in \tau _{x_{0}}}\forall _{\alpha \in A}\forall _{x,\,y\in V}\
T_{\alpha }x-T_{\alpha }y\in B,
\end{equation*}%
where $\tau _{x_{0}}$ denotes $\tau $-neighbourhoods of $x_{0}$. With $V$
satisfying the above condition we have $Tx-Ty\in \bar{B}\subset U$. Thus
\begin{equation*}
\forall _{U\in \tau _{0}}\exists _{V\in \tau _{x_{0}}}\forall _{x,\,y\in
V}Tx-Ty\in U
\end{equation*}%
which means that each $T\in \mathcal{\bar{S}}$ is $\tau $-continuous at any $%
x_{0}\in C$. In particular, it relates to $r$.
\end{proof}

We are now ready to state the following qualitative version of the
Markov-Kakutani theorem.

\begin{corollary}
Suppose $\mathcal{S}$ is a commutative family of $\tau $-continuous affine
mappings on $C$, where $C$ is a convex and $\tau $-compact bounded subset of
a Banach space. Then there exists an affine, surjective idempotent $%
r:C\rightarrow \Fix\mathcal{S}.$ If, moreover, elements of $\mathcal{S}$ are
locally $\tau $-equicontinuous, then there exists a ($\tau $-continuous)
retraction onto $\Fix\mathcal{S}$.
\end{corollary}

\begin{proof}
Construct from $\mathcal{S}$ a semigroup $\hat{\mathcal{S}}$ generated by
averaged mappings $\{\frac{1}{2}I+\frac{1}{2}T:T\in \mathcal{S}\}$ which is
obviously commutative and consists of affine $\tau $-continuous mappings.
Recall that a commutative semigroup is subsurjective on compact sets. Now it
is enough to notice that $\hat{\mathcal{S}}$ is locally $\tau $%
-equicontinuous whenever $\mathcal{S}$ is, and apply Theorem \ref{Gen}.
\end{proof}

\begin{remark}
\label{weak}Note that in the case of $w$-topology the assumption about local
$w$-equicontinuity can be changed to (strong) local equicontinuity because
every strongly continuous affine mapping is weakly continuous. Even in this
case the result seems to have been known only in strictly convex Banach
spaces (see \cite[Theorem 5.8]{Ph}).
\end{remark}

\begin{remark}
The results of this section are formulated for subsets of Banach spaces only
but similar arguments apply to the case of locally convex spaces, as well.
\end{remark}

Note, on the margin, that starting again from Lemma \ref{CFP} and basically
repeating the scheme presented in this section, we get a \textquotedblleft
firmly nonexpansive\textquotedblright\ counterpart of Theorem \ref{Gen}.

\begin{theorem}
\label{Gen2}Suppose $\mathcal{S}$ is a $\tau $-properly represented
semigroup generated by firmly nonexpansive mappings on $C$, where $C$ is a $%
\tau $-compact bounded subset of a Banach space. Then there exists in $%
\mathcal{\bar{S}}$ an idempotent mapping $r$ from $C$ onto $\Fix\mathcal{S}$%
. If, moreover, the norm is $\tau $-lower semicontinuous, then $\Fix\mathcal{%
S}$ is a nonexpansive retract of $C$.
\end{theorem}

Even more generally,

\begin{theorem}
\label{Gen3}Suppose $\mathcal{S}$ is a properly represented semigroup
generated by UAR mappings on a compact metric space $M$. Then there exists
in $\mathcal{\bar{S}}$ a surjective idempotent $r:M\rightarrow \Fix\mathcal{S%
}$.
\end{theorem}


\begin{thebibliography}{99}
\bibitem{AnIs} K. Anzai, S. Ishikawa, On common fixed points for several
continuous affine mappings, Pacific Journal of Mathematics 72 (1977), 1--4.

\bibitem{ALM} D. Ariza-Ruiz, G. L\'{o}pez-Acedo, V. Mart\'{\i}n-M\'{a}rquez,
Firmly nonexpansive mappings, J. Nonlinear Convex Anal. 15 (2014), 61--87.

\bibitem{BGM} U. Bader, T. Gelander, N. Monod, A fixed point theorem for $%
L_{1}$ spaces, Invent. Math. 189 (2012), 143--148.

\bibitem{BBR} J. B. Baillon, R. E. Bruck, S. Reich, On the asymptotic
behavior of nonexpansive mappings and semigroups in Banach spaces, Houston
J. Math. 4 (1978), 1--9.

\bibitem{BoWi} S. Borzdy\'{n}ski, A.Wi\'{s}nicki, A common fixed point
theorem for a commuting family of weak$^{\ast }$ continuous nonexpansive
mappings, Studia Mathematica 225 (2014), 173--181.

\bibitem{BrPe} F. E. Browder, W. V. Petryshyn, The solution by iteration of
nonlinear functional equations in Banach spaces, Bull. Amer. Math. Soc. 72
(1966), 571--575.

\bibitem{Br2} R. E. Bruck, Jr., A common fixed point theorem for a commuting
family of nonexpansive mappings, Pacific J. Math. 53 (1974), 59--71.

\bibitem{BrRe} R. E. Bruck, S. Reich, Nonexpansive projections and
resolvents of accretive operators in Banach spaces, Houston J. Math. 3
(1977), 459--470.

\bibitem{EdOb} M. Edelstein, R. C. O'Brien, Nonexpansive mappings,
asymptotic regularity and successive approximations, J. London Math. Soc.
(2) 17 (1978), no. 3, 547--554.

\bibitem{GoKi3} K. Goebel, W. A. Kirk, Iteration processes for nonexpansive
mappings, in: Topological methods in nonlinear functional analysis, S. P.
Singh, S. Thomeier, B. Watson (eds.), AMS, Providence, R.I., 1983, 115--123.

\bibitem{GoKi} K. Goebel, W. A. Kirk, Topics in metric fixed point theory,
Cambridge University Press, Cambridge, 1990.

\bibitem{GrDu} A. Granas, J. Dugundji, Fixed point theory, Springer-Verlag,
New York, 2003.

\bibitem{HeRo} E. Hewitt, K. A. Ross, Abstract harmonic analysis. Vol. I:
structure of topological groups. Integration theory, group representations,
Academic Press, Publishers, New York, 1963.

\bibitem{HoLa} R. D. Holmes, A. T. Lau, Non-expansive actions of topological
semigroups and fixed points, J. London Math. Soc. 5 (1972), 330--336.

\bibitem{Is} S. Ishikawa, Fixed points and iteration of a nonexpansive
mapping in a Banach space. Proc. Amer. Math. Soc. 59 (1976), no. 1, 65--71.

\bibitem{Kra} M. A. Krasnoselskii, Two observations about the method of
successive approximations, Uspehi Math. Nauk. 10 (1955), 123--127.

\bibitem{Lau1} A. T.-M. Lau, Amenability and fixed point property for
semigroup of nonexpansive mappings, in: Fixed Point Theory and Applications,
S. Swaminathan (ed.), Academic Press, New York-London, 1976, 121--129.

\bibitem{LaTa} A. T.-M. Lau, W. Takahashi, Invariant means and fixed point
properties for non-expansive representations of topological semigroups,
Topol. Methods Nonlinear Anal. 5 (1995), 39--57.

\bibitem{LaTa2} A. T.-M. Lau, W. Takahashi, Fixed point and non-linear
ergodic theorems for semigroups of non-linear mappings, in: Handbook of
Metric Fixed Point Theory, W. A. Kirk, B. Sims (eds.), Kluwer Academic
Publishers, Dordrecht, 2001, 515--553.

\bibitem{LaZh} A.T.-M. Lau, Y. Zhang, Fixed point properties for semigroups
of nonlinear mappings and amenability. J. Funct. Anal. 263 (2012),
2949--2977.

\bibitem{Ph} V.Q. Ph\'{o}ng, Nonlinear almost periodic actions of
semigroups, in: Functional analysis, K.D. Bierstedt at al. (eds.), Marcel
Dekker, 1994, 71--94.

\bibitem{ReSh} S. Reich, I. Shafrir, The asymptotic behavior of firmly
nonexpansive mappings. Proc. Amer. Math. Soc. 101 (1987), 246--250.

\bibitem{HoKu} H. K. Xu, I. Yamada, Asymptotic regularity of linear power
bounded operators, Taiwanese J. Math. 10 (2006), 417--429.
\end{thebibliography}
\end{document}